\documentclass[12pt,epsfig,amsfonts]{amsart}
\usepackage{amsmath,amsthm,amssymb,mathrsfs,amscd,epsfig,color}
\usepackage{url}
\usepackage{graphicx}
\usepackage{mathrsfs}
\usepackage{ulem}
\usepackage{fancybox}
\usepackage{pb-diagram} 
\usepackage[all]{xy}
\usepackage{comment}
\usepackage{mathtools}
\usepackage{centernot}
\usepackage{placeins} 
\usepackage{url}
\usepackage{fancybox}

\makeatletter
\newcommand{\tpitchfork}{%
  \vbox{
    \baselineskip\z@skip
    \lineskip-.52ex
    \lineskiplimit\maxdimen
    \m@th
    \ialign{##\crcr\hidewidth\smash{$-$}\hidewidth\crcr$\pitchfork$\crcr}
  }%
  
}
\makeatother

\newtheorem{prop}{Proposition}[section]
\newtheorem{lemma}[prop]{Lemma}

\newtheorem{theorem}[prop]{Theorem}

\theoremstyle{remark}
\newtheorem{remark}[prop]{Remark}

\usepackage{mathtools}
\usepackage{ifmtarg}

    \newcommand\contFrac{\@ifstar{\@contFracStar}{\@contFracNoStar}}

   \def\singleContFrac#1#2{%
        \begin{array}{@{}c@{}}%
            \multicolumn{1}{c|}{#1}%
            \\%
            \hline%
           \multicolumn{1}{|c}{#2}%
        \end{array}%
   }

    \def\@contFracNoStar#1{%
        \mathchoice{
            \@contFracNoStarDisplay@#1//\@nil%
        }{
            \@contFracNoStarInline@#1//\@nil%
        }{
            \@contFracNoStarInline@#1//\@nil%
        }{
            \@contFracNoStarInline@#1//\@nil%
        }%
    }

    \def\@contFracNoStarDisplay@#1//#2\@nil{%
        \@ifmtarg{#2}{%
            #1%
        }{%
            #1+\cfrac{1}{\@contFracNoStarDisplay@#2\@nil}%
        }%
    }

        \def\@contFracNoStarInline@#1//#2\@nil{%
            \@ifmtarg{#2}{%
                #1%
            }{%
                #1 \@@contFracNoStarInline@@#2\@nil%
            }%
        }
        \def\@@contFracNoStarInline@@#1//#2\@nil{%
            \@ifmtarg{#2}{%
                + \singleContFrac{1}{#1}%
            }{%
                + \singleContFrac{1}{#1} \@@contFracNoStarInline@@#2\@nil%
            }%
        }

    \def\@contFracStar#1{%
        \mathchoice{
            \@contFracStarDisplay@#1////\@nil%
        }{
            \@contFracStarInline@#1//\@nil%
        }{
            \@contFracStarInline@#1//\@nil%
        }{
            \@contFracStarInline@#1//\@nil%
        }%
    }

    \def\@contFracStarDisplay@#1//#2//#3\@nil{%
        \@ifmtarg{#2}{%
            #1%
        }{%
            #1 + \cfrac{#2}{\@contFracStarDisplay@#3\@nil}%
        }%
    }

        \def\@contFracStarInline@#1//#2\@nil{%
            \@ifmtarg{#2}{%
                #1%
            }{%
                #1 \@@contFracStarInline@@#2\@nil%
            }%
        }
        \def\@@contFracStarInline@@#1//#2//#3\@nil{%
            \@ifmtarg{#3}{%
                + \singleContFrac{#1}{#2}%
            }{%
                + \singleContFrac{#1}{#2} \@@contFracStarInline@@#3\@nil%
            }%
        }

\numberwithin{equation}{section}

\newcommand{\confrac}[2]{%
  \frac{\displaystyle{%
    \strut\hfill{#1}\hfill\;\vrule}}%
      {\displaystyle{%
       \strut\vrule\;\hfill{#2}\hfill}}}%
\usepackage[useregional]{datetime2}

\title[The distribution of the largest digit for parabolic IFSs]{The distribution of the largest digit for\\ parabolic Iterated Function Systems\\ of the interval}
\author{Hiroki Takahasi}

\address{Keio Institute of Pure and Applied Sciences (KiPAS), Department of Mathematics,
Keio University, Yokohama,
223-8522, JAPAN} 
\email{hiroki@math.keio.ac.jp}


\begin{document}

\begin{abstract}
We investigate the distribution of the largest digit for a wide class of infinite parabolic Iterated Function Systems (IFSs) of the unit interval. Due to the recurrence to parabolic (neutral) fixed points, the dimension analysis of these systems become more delicate than that of uniformly contracting IFSs.
We show that the Hausdorff dimensions of level sets associated with the largest digits are constantly equal to the Hausdorff dimension of the limit set of the IFS. This result is an analogue of Wu and Xu's theorem [{\it Math. Proc. Cambridge Philos. Soc.} {\bf 146} (2009), no.1, 207--212] on the regular continued fraction. Examples of application of our result include the backward (aka minus, or negative) continued fractions, even-integer continued fractions, and go  
beyond. 
Our main tool is a dimension theory for non-uniformly expanding Bernoulli interval maps with infinitely many branches. 
\end{abstract}
\maketitle

\section{Introduction}
Any irrational number $x$ in $(0,1)$
   has a unique infinite expansion of the form 
 \[
x=\confrac{1 }{a_{1}(x)} + \confrac{1 }{a_{2}(x)}  +\confrac{1 }{a_{3}(x)}+\cdots,\]
called {\it the regular continued fraction},
  where each digit $a_n(x)$
 belongs to the set $\mathbb N$ of positive integers.
For a typical irrational $x\in (0,1)$
in the sense of the Lebesgue measure, 
the sequence $\{a_n(x)\}_{n\in\mathbb N}$ is unbounded and consists of mostly small digits, punctuated by occasional larger ones \cite{H00}. It is clear that
 the occurrence of large digits influences  statistical properties of digit sequences, see e.g.,
\cite{DB86} for details. Therefore, it is important to investigate 
the growth of the largest digit
$L_n(x)=\max\{a_1(x),\ldots,a_n(x)\}$ as $n\to\infty$. A pioneering result in this direction is due to
Galambos \cite{Gal72} who proved that 
\[\lim_{n\to\infty}\mu_{\rm G}\left\{x\in(0,1)\setminus\mathbb Q\colon\frac{L_n(x) }{n}<\frac{y}{\log2}\right\}=e^{-\frac{1}{y}}\ \text{ for }y>0,\]
where $\mu_{\rm G}$ denotes the Gauss measure  $\frac{1}{\log2}\frac{dx}{1+x}$. Later in \cite{Gal74} he proved that
\[\lim_{n\to\infty}\frac{\log L_n(x)}{\log n}=1\ \text{ for Lebesgue a.e. }x\in(0,1)\setminus\mathbb Q.\]
Philipp \cite{Phi75} considered the number
\[\ell(x)=\liminf_{n\to\infty}\frac{
L_n(x)\log\log n}{n},\]
and proved that $\ell(x)=1/\log2$ for Lebesgue a.e. $x\in(0,1)\setminus\mathbb Q$, solving
 Erd\H{o}s' conjecture apart from the value.
 Okano \cite{Oka02} proved that 
 for any $\alpha\in(0,\infty)$ there exists $x\in(0,1)\setminus\mathbb Q$ such that $\ell(x)=\alpha$. 
 Wu and Xu \cite[Theorem~1.1]{WuXu09} significantly strengthened Okano's result, proving that 
 for any $\alpha\in[0,\infty)$ the set
\[\left\{x\in(0,1)\setminus\mathbb Q\colon\lim_{n\to\infty}\frac{L_n(x)\log\log n}{n}=\alpha\right\}\]
is of Hausdorff dimension $1$. Chang and Chen \cite{CC18} proved results analogous to that of Wu and Xu replacing $L_n(x)$ or the norming function $n/\log\log n$.

All these interesting developments have taken place exclusively for the regular continued fraction.
The aim of this paper is to investigate growth rates of the largest digits for a wide class of infinite parabolic Iterated Function Systems. As a corollary we obtain an analogue of Wu and Xu's theorem for other continued fractions with totally different Lebesgue typical behaviors.

Let $X$ be a compact interval with positive Euclidean diameter.
Let $I$ be a subset of $\mathbb N$ with $\#I\geq2$, and 
let  $\phi_i\colon X\to X$ $(i\in I)$ be $C^1$ maps.
The collection
$\Phi=\{\phi_i\}_{i\in I }$ is called  {\it an Iterated Function System (IFS)} on $X$. It is called {\it an infinite (resp. finite)} IFS if $I$ is an infinite (resp. finite) set.  We say an IFS $\Phi$ satisfies {\it the open set condition} if for all distinct indices $i,j\in I,$ 
\[
\phi_i({\rm int}X)\cap \phi_j({\rm int}X)=\emptyset.
\]

Let $\Phi=\{\phi_i\}_{i\in I}$
be an IFS on $X$. For $\omega=(\omega_1,\omega_2,\ldots)\in I^{\mathbb N}$ and $n\in \mathbb{N}$, we set
\[\phi_{\omega_1\cdots \omega_n}=\phi_{\omega_1}\circ\cdots\circ\phi_{\omega_n}.\]
If the set $\bigcap_{n=1}^{\infty}\phi_{\omega_1\cdots\omega_{n}}(X)$ is a singleton for any $\omega\in I^{\mathbb N}$,  we
 define {\it an address map} 
 $\Pi \colon I^{\mathbb N} \to X$ by
\[\Pi(\omega)\in \bigcap_{n=1}^{\infty}\phi_{\omega_1\cdots\omega_n}(X),\]
and {\it the limit set} of $\Phi$ by
 \[
\Lambda(\Phi)=\Pi( I^{\mathbb N} ).
\]
Since the address map may not be injective, we introduce the set \[\Lambda_*(\Phi)=\{x\in\Lambda(\Phi)\colon\#\Pi^{-1}(x)=1\}.\] 
Since $X$ is an interval, if the open set condition holds then
 $\Lambda(\Phi)\setminus\Lambda_*(\Phi)$ is countable and of Hausdorff dimension zero.
For each $x\in\Lambda_*(\Phi)$,
 there is a unique sequence $\{a_n(x)\}_{n=1}^\infty\in I^\mathbb N$  
that satisfies $x=\Pi(\{a_n(x)\}_{n=1}^\infty)$.
Note that
\[
x = \lim_{n\to\infty} \phi_{a_1(x)} \circ \cdots \circ \phi_{a_n(x)} (y)\ \text{ for all }y\in X.
\]
If $\Phi$ is an infinite IFS, then
for each $\alpha\in[0,\infty]$, define a level set
\[L(\alpha)=\left\{x\in\Lambda_*(\Phi)\colon\lim_{n\to\infty}\frac{\max\{a_1(x),\ldots,a_n(x)\}\log\log n}{n}=\alpha\right\}.\]
Moreover, if there exist constants $c>0$, $d>1$ such that
\[\min_{x\in X}|\phi'_i(x)|\geq \frac{c}{i^d}
\ \text{ for every }i\in I,\]
then we say $\Phi$ is {\it $d$-decaying}.

An IFS $\Phi$ on $X$ is called {\it parabolic}
if the open set condition holds, and the following two conditions hold:
\begin{itemize}

\item[(A1)] (Non-uniform contraction) $|\phi'_i(x)|<1$ everywhere except for finitely many pairs $(i,x_i)$, $i\in I$, for which $x_i$ is the unique fixed point of $\phi_i$ and $|\phi'_i(x_i)|=1$.
Such pairs and indices $i$ are called {\it parabolic}.

\item[(A2)] (Bounded distortion) 
There exists a constant $C\ge 1$ such that for all $\omega\in I^{\mathbb N}$ and 
 $n\in\mathbb N_{\geq2}$ such that $\omega_{n}$ is not a parabolic index, or else $\omega_{n-1}\neq\omega_n$, 
\[
|\phi'_{\omega_1\cdots\omega_n}(x)|\le C|\phi'_{\omega_1\cdots\omega_n}(y)|\ \text{ for all }x,y\in X.
\]

\end{itemize}

For results on a dimension theory of finite or infinite parabolic IFSs, see \cite{GelRam09,Iom10,JJOP10,MauUrb00,MauUrb03,Nak00,Urb96} for example. Here,
our interest is in infinite parabolic IFSs. A prime example is given by the backward (aka minus, or negative) continued fraction, see Section~\ref{verify-s}
for more details.
Our definition of parabolic IFS is simpler than that in \cite[Section~8]{MauUrb03} since we only work on IFSs on a compact interval.

Let $\dim_{\rm H}$ denote the Hausdorff dimension on $[0,1]$. Our main result is stated as follows.

\begin{theorem}\label{mainthm} Let $d>1$, and let $\Phi=\{\phi_i\}_{i\in I}$ be a $d$-decaying parabolic IFS on $[0,1]$ that satisfies (B1) (B2). For any 
 $\alpha\in[0,\infty]$ we have
\[\dim_{\rm H}L(\alpha)=\dim_{\rm H}\Lambda(\Phi).\]
\end{theorem}

Conditions (B1) (B2) will be given in Section~\ref{ass}.
For any infinite parabolic IFS $\Phi$ on $[0,1]$, the set $\bigcap_{n=1}^{\infty}\phi_{\omega_1\cdots\omega_{n}}([0,1])$ is a singleton for any $\omega\in I^{\mathbb N}$ (see Lemma~\ref{unif-decay}). Hence, the limit set $\Lambda(\Phi)$ is defined. A key observation is that the Hausdorff dimension of the set
\[\{x\in\Lambda_*(\Phi)\colon a_n(x)\text{ is not parabolic for every $n\in\mathbb N$}\}\]
becomes strictly smaller than that of $\Lambda(\Phi)$ when $\Phi$ has a parabolic index. This means that we cannot exclude parabolic indices from consideration to establish the equality in Theorem~\ref{mainthm}.

To prove Theorem~\ref{mainthm}, it suffices to show the inequality $\dim_{\rm H}L(\alpha)\geq\dim_{\rm H}\Lambda(\Phi)$. To this end, we first extract 
from $\Phi$
a family of finite IFSs whose limit sets have Hausdorff dimension arbitrarily close to $\dim_{\rm H}\Lambda(\Phi)$.
We then construct a subset of $L(\alpha)$
 by inserting large digits into each limit set
without substantially losing Hausdorff dimension. 
This two-step construction has been 
inspired by the above-mentioned work of Wu and Xu \cite{WuXu09}. 
However, 
if $\Phi$ has parabolic indices then
the construction of a family of finite IFSs and the selection of positions of digits to be inserted have to be done carefully, in order to avoid the influence of parabolic indices on distortion estimates. We assume (B1) (B2) to facilitate these issues. 
These two conditions are translations from the paper \cite{JT} by Jaerisch and the author. For more details, see Section~\ref{ass}.

The rest of this paper is organized as follows.
In Section~2 we provide preliminary materials, including the definitions of (B1) (B2) and the construction of a family of finite IFSs. 
In Section~3 we complete the proof of  Theorem~\ref{mainthm}.
In Section~\ref{verify-s} we give a verifiable sufficient condition for (B1) (B2), and provide examples of parabolic IFSs satisfying them. 
\section{Preliminaries}
This section provides preliminary materials.
 In Section~\ref{ass} we introduce conditions (B1) (B2). In Section~\ref{Markov-sec} we translate a few results in \cite{JT} into the language of IFS. In Section~\ref{large-sec} we construct a family of finite IFSs with large limit sets, from a given parabolic IFS satisfying (B1) (B2).
\subsection{Saturation and mild distortion}\label{ass}
Let $\Phi=\{\phi_i\}_{i\in I}$ be a parabolic IFS. For each $n\in\mathbb N$ let $I^n$ denote the set of words from  
$I$ with word length $n$. 
Let $\mathcal M$ denote the set of shift invariant ergodic 
 Borel probability measures on the Cartesian product topological space $I^{\mathbb N}$. For each $\nu\in\mathcal M$, 
 define {\it the Lyapunov exponent} of $\nu$ 
 by 
 \[\chi(\nu)=-\int\log|\phi'_{\omega_1 }(\Pi(\omega))| d\nu(\omega)\in[0,\infty],\]
 and 
set \[\mathcal M(\Phi)=\left\{\nu\in \mathcal M\colon \chi(\nu)<\infty\right\}.\] For a Borel probability measure $\mu$ on $[0,1]$, define \[\dim(\mu)=\inf\{\dim_{\rm H}A\colon A\subset[0,1],\ \mu(A)=1\}.\]
\begin{itemize}
\item[(B1)] (Saturation)
  $\dim_{\rm H}\Lambda(\Phi)=\sup\left\{\dim(\nu\circ\Pi^{-1})\colon \nu\in\mathcal M(\Phi)\right\}.$
\end{itemize}

For each $n\in\mathbb N$
we define \[D_n(\Phi)=\sup_{\omega_1\cdots \omega_n\in I^n }\max_{x,y\in [0,1] }\log\frac{\phi'_{\omega_1\cdots \omega_n}(x)}{\phi'_{\omega_1\cdots \omega_n}(y)}.\]
\begin{itemize}
\item[(B2)] (Mild distortion)
 $D_1(\Phi)<\infty$ and  $D_n(\Phi)=o(n)$. 
\end{itemize}

The open set condition allows us to convert an IFS $\Phi$ on $[0,1]$ to iterations of a Bernoulli map on $[0,1]$, see Section~\ref{Markov-sec} for the definition.
Each neutral fixed point of this map corresponds to a parabolic index of $\Phi$.
This correspondence allows us to import results from the paper \cite{JT} on a dimension theory of non-uniformly expanding one-dimensional Markov maps with infinitely many branches. 
Conditions (B1) (B2) are translations from \cite{JT} into the language of IFS.
A verifiable sufficient condition for (B1) (B2) will be given in Proposition~\ref{sat-prop}.

\begin{remark}
The regular continued fraction is generated by the $2$-decaying IFS $\Phi=\{\phi_{i}\}_{i\in\mathbb N}$ on $[0,1]$ given by $\phi_i(x)=1/(x+i)$, which is a parabolic IFS without parabolic indices with
 $\Lambda(\Phi)=\Lambda_*(\Phi)=(0,1)\setminus\mathbb Q$:
 $x=\lim_{n\to\infty}\phi_{a_1(x)}\cdots\phi_{a_n(x)}(0)$ for all $x\in\Lambda(\Phi)$. 
 Condition (B1) holds because $\dim(\mu_{\rm G})=1$.
 Using the uniform contraction $\sup_{\omega\in\mathbb N^2}\max_{x\in[0,1]}|\phi_{\omega}'(x)|<1$ and so-called R\'enyi's condition, one can verify $\sup_{n\in\mathbb N}D_n(\Phi)<\infty$ which is stronger than (B2). As a result,
one can recover \cite[Theorem~1.1]{WuXu09} by applying Theorem~\ref{mainthm} to this IFS $\Phi$.\end{remark}

 \subsection{Entropy, dimension, decay of fundamental intervals}\label{Markov-sec}
We say $f\colon\varDelta\to [0,1]$ is {\it a Markov map}
  if the following three conditions hold:
 
 \begin{itemize}
 
\item[(M0)]
   there exist a subset $I$ of $\mathbb N$ and a family $\{\varDelta_i\}_{i\in I}$ 
of pairwise disjoint non-empty open intervals in $(0,1)$
such that $\varDelta=\bigcup_{i\in I}\varDelta_i$;

 \item[(M1)] for each $i\in I$, the restriction  
 $f|_{\varDelta_i}$
extends to a $C^1$ diffeomorphism from the closure of $\varDelta_i$ onto its image; 
\item[(M2)] $f(\varDelta_i)\cap\varDelta_j\neq\emptyset$ for $i,j\in I$ implies $\Delta_j\subset f(\varDelta_j)$.
\end{itemize}
We say $f\colon\Delta\to[0,1]$ is {\it a Bernoulli map} if (M0) (M1), and the following holds: 
\begin{itemize}
 \item[(M3)]for every $i\in I$, 
  $f(\varDelta_i)=(0,1)$.
  \end{itemize}

 Given a parabolic IFS $\Phi=\{\phi_i\}_{i\in I}$ on $[0,1]$, we can associate
  {\it a Bernoulli map} $f\colon \bigcup_{i\in I}\phi_i((0,1))\to(0,1)$ by
 $f|_{\phi_i((0,1))}=\phi_i^{-1}|_{\phi_i((0,1))}$.
 The parabolic indices of $\Phi$ correspond to neutral fixed points of $f$,  $\Lambda(\Phi)$ contains $\bigcap_{n=0}^\infty f^{-n}(\varDelta)$, and $\Lambda(\Phi)\setminus\bigcap_{n=0}^\infty f^{-n}(\varDelta)$ is a countable set.
 For each $\nu\in\mathcal M$, the measure $\nu\circ\Pi^{-1}$ is $f$-invariant. Let $h(\nu)$ denote the measure-theoretic entropy of $\nu\circ\Pi^{-1}$ with respect to $f$. We say $\nu$ is {\it ergodic} if 
 $\nu\circ\Pi^{-1}$ is ergodic with respect to $f$,
 and is {\it expanding} if $\chi(\nu)>0$.
If $\nu\in\mathcal M(\Phi)$ is ergodic and expanding, then 
\begin{equation}\label{dim-formula}\dim(\nu\circ\Pi^{-1})=\frac{h(\nu)}{\chi(\nu)},\end{equation}
see e.g., \cite{JT} and \cite[Section~4.4]{MauUrb03} for details.

Let $\Phi=\{\phi_i\}_{i\in I}$ be a parabolic IFS on $[0,1]$.
For each $n\in\mathbb N$
and $\omega=(\omega_1,\ldots,\omega_n)\in I^n$,
we call the closed interval
\[J(\omega)=J(\omega_1,\ldots,\omega_n)=\phi_{\omega_1\cdots \omega_n}([0,1])\]
 {\it a fundamental interval} of order $n$.
For convenience, let us call $[0,1]$ a fundamental interval of order $0$.
Fundamental intervals of the same order are either disjoint, coincide or intersect only at their boundary points.

We say $\Phi$ has
{\it decay of fundamental intervals} if $\bigcap_{n=1}^{\infty}\phi_{\omega_1\cdots\omega_{n}}([0,1])$ is a singleton for any $\omega\in I^{\mathbb N}$.
 We say $\Phi$ has {\it uniform decay of fundamental intervals} 
if
\[\displaystyle{\lim_{n\to\infty}}\sup_{\omega\in I^n}|J(\omega)|=0.\]

\begin{lemma}\label{unif-decay}
Every parabolic IFS $\Phi$ has decay of fundamental intervals.
If moreover $D_1(\Phi)<\infty$, then $\Phi$ has uniform decay of fundamental intervals.
\end{lemma}
\begin{proof} This follows from \cite[Proposition~3.1]{JT} applied to the Bernoulli map  associated with the IFS $\Phi$ in the lemma. 
For each $n\in\mathbb N$, fundamental intervals of order $n$ are the closures of 
the $n$-cylinders in the language of \cite{JT}. \end{proof}

  \begin{lemma}\label{katok-lem}
    Let $\Phi=\{\phi_i\}_{i\in I}$ be a parabolic IFS satisfying (B1) (B2).
  Let
$\nu\in\mathcal M(\Phi)$ be ergodic and expanding.
For any $\varepsilon>0$ there exists 
 $p_0\geq2$ such that for every integer $p\geq p_0$
there exists
a finite subset $V_p$ of $I^{p}$ such that 
\[
\left|\frac{1}{p}\log\#V_p- h(\nu)\right|<\varepsilon\ \text{ 
and }\
\max_{\omega\in V_p}\max_{x\in[0,1] }\left|\frac{1}{p}\log |\phi_\omega'(x)|^{-1}-\chi(\nu)\right|<\varepsilon.\]
\end{lemma}
\begin{proof}This follows from the proof of \cite[Lemma~3.5]{JT} applied to the Bernoulli map  associated with the IFS $\Phi$ in the lemma. \end{proof}

\subsection{Construction of finite IFSs with large limit sets}\label{large-sec}
The next proposition provides a family of finite IFSs without parabolic indices
whose limit sets approximate that of the original parabolic IFS in terms of Hausdorff dimension.
\begin{prop}\label{katok} Let $\Phi=\{\phi_i\}_{i\in I}$ be a parabolic IFS on $[0,1]$ that satisfies (B1) (B2).
For any $\varepsilon>0$ there exist 
 $p_0\geq2$ such that for every integer $p\geq p_0$ there exist a constant $\gamma>0$, 
a non-empty finite subset $W_{p}$ of $I^{p}$ with the following properties:

\begin{itemize}

\item[(a)] for all distinct words $\omega,\eta\in W_{p}$, $J(\omega)\cap J(\eta)=\emptyset;$

\item[(b)] for any $\omega=\omega_1\cdots \omega_{p}\in  W_{p}$, $\omega_p$ is not a parabolic index;

\item[(c)] for any $\omega\in W_{p}$, $\max_{x\in[0,1]}|\phi'_{\omega}(x)|<e^{-\gamma p}$;

\item[(d)]
for the finite IFS
$\Phi_{{p}}=\{\phi_{\omega} \}_{\omega\in 
W_{p} }$ on $[0,1]$ without parabolic indices, 
\[\dim_{\rm H}\Lambda(\Phi_{{p} })>\dim_{\rm H}\Lambda(\Phi)-\varepsilon.\]
\end{itemize}
\end{prop}

\begin{proof}Fix an index $a\in I$ that is not parabolic. By (B1) and \eqref{dim-formula}, for any $\varepsilon>0$ there exists an ergodic expanding measure $\nu\in \mathcal M(\Phi)$ such that \begin{equation}\label{katok-0}\frac{h(\nu)}{\chi(\nu)}>\dim_{\rm H}\Lambda(\Phi)-\frac{\varepsilon}{2}.\end{equation}
Let $\delta\in(0,h(\nu))$. By Lemma~\ref{katok-lem},
for all sufficiently large integer $p\geq2$ there exists a non-empty finite subset $V_{p-1}$ of $I^{p-1}$ such that 
\begin{equation}\label{katok-eq1}\left|\frac{1}{p-1}\log\#V_{p-1}-h(\nu)\right|<\frac{\delta}{2},\end{equation} and for every $\omega\in V_{p-1}$,
\begin{equation}\label{katok-eq2}\max_{x\in [0,1] }\left|\frac{1}{p-1}\log |\phi_\omega'(x)|^{-1}-\chi(\nu)\right|<\frac{\delta}{2}.\end{equation}
The fundamental intervals
in the collection $\{J(\omega_1,\ldots,\omega_{p-1},a)\colon \omega_1\cdots \omega_{p-1}\in V_{p-1}\}$ 
may not be pairwise disjoint, intersecting each other at their endpoints. We remove from this collection every other interval with respect to the natural order in $[0,1]$. 
What is left is the collection of pairwise disjoint fundamental intervals of order $p$ containing at least $\lfloor(\#V_{p-1})/2\rfloor$ elements. Define $W_p$ to be the collection of words in $I^p$ that correspond to these remaining intervals. 
  Clearly (a) (b) hold.
  For all sufficiently large $p$,
 the estimates in \eqref{katok-eq1} and \eqref{katok-eq2} remain intact. Indeed we have
\begin{equation}
\begin{split}\label{katok-eq3}\left|\frac{1}{p}\log\#W_p-h(\nu)\right|\leq&\frac{1}{p(p-1)}\log\#W_{p}+\frac{1}{p-1}\left|\log\#W_{p-1}-\log\#V_{p-1}\right|\\
&+\left|\frac{1}{p-1}\log\#V_{p-1}-h(\nu)\right|<\delta,\end{split}\end{equation}
and for every $\omega\in W_p$, \begin{equation}
\begin{split}\label{katok-eq4}\max_{ x\in[0,1]}\left|\frac{1}{p}\log |\phi'_\omega(x)|^{-1}-\chi(\nu)\right|&\leq\frac{1}{p(p-1)}\max_{ x\in[0,1] }\log |\phi'_\omega(x)|^{-1}\\&+\max_{ x\in[0,1] }\left|\frac{1}{p-1}\log |\phi'_\omega(x)|^{-1}-\chi(\nu)\right|<\delta.\end{split}\end{equation}
 Item (c) follows from \eqref{katok-eq4}
 by setting $\gamma=\chi(\nu)-\delta>0$.

For convenience we set $I^0=\{0\}$, and set $\phi_{0}$ to be the identity map on $[0,1]$.
Let $\nu_{\rm max}$ denote the measure of maximal entropy for the restriction of the $p$-iteration of the associated Bernoulli map to
$\Lambda(\Phi_p)$. The measure 
$\tilde\nu=(1/p)\sum_{i=0}^{p-1}\sum_{\omega\in I^i}\nu_{\rm max}\circ\phi_{\omega}$ belongs to $\mathcal M(\Phi)$, ergodic, expanding and 
so $\dim(\tilde\nu)=h(\tilde\nu)/\chi(\tilde\nu)$ by \eqref{dim-formula}. Moreover we have 
\begin{equation}\label{katok-eq5}h(\tilde\nu)=\frac{1}{p}\log\#W_p\ \text{ and }\ \chi(\tilde\nu)\leq \frac{1}{p}\max_{\omega\in W_p}\max_{x\in[0,1]}|\phi'_{\omega}(x)|^{-1}.\end{equation}
Combining \eqref{katok-eq3}, \eqref{katok-eq4} and
\eqref{katok-eq5} yields
\[\dim(\tilde\nu)=\frac{h(\tilde\nu)}{\chi(\tilde\nu)}>\frac{h(\nu)-\delta}{\chi(\nu)+\delta}.\]
Since the set 
$\bigcup_{i=0}^{p-1}\bigcup_{\omega\in I^i}\phi_\omega(\Lambda(\Phi_{p}))$ has full
$\tilde\nu$-measure, we have
\[\dim_{\rm H}\left(\bigcup_{i=0}^{p-1}\bigcup_{\omega\in I^i}\phi_\omega(\Lambda(\Phi_{p}))\right)\geq\dim(\tilde\nu).\] Combining these two inequalities yields
 \begin{equation}\label{katok-fin}\dim_{\rm H}\Lambda(\Phi_{p})=\dim_{\rm H}\left(\bigcup_{i=0}^{p-1}\bigcup_{\omega\in I^i}\phi_\omega(\Lambda(\Phi_{p}))\right)>\frac{h(\nu)-\delta}{\chi(\nu)+\delta}.\end{equation}
Since $\delta>0$ is arbitrary, combining \eqref{katok-0}, \eqref{katok-fin} and then reducing $\delta$ if necessary we obtain the desired inequality in (d). \end{proof}

\section{Proof of the main result}
In this section we complete the proof of Theorem~\ref{mainthm}. 
\subsection{Initial setup}
Let $d>1$, and let $\Phi=\{\phi_i\}_{i\in I}$ be a $d$-decaying parabolic IFS on $[0,1]$ that satisfies (B1) (B2). To prove Theorem~\ref{mainthm} it suffices to show
$\dim_{\rm H}L(\alpha)\geq\dim_{\rm H}\Lambda(\Phi)$ for any $\alpha\in[0,\infty]$.
Let $\varepsilon>0$. Let $p\geq2$ be a sufficiently large integer such that
\begin{equation}\label{p} (k+p+1)^d+2p+1<(k+p+2)^d\ \text{ for every }k\in\mathbb N\cup\{0\}.\end{equation}
Let $W_p$ be 
a non-empty finite subset of $I^{p}$, and let $\gamma>0$ be a constant for which the conclusions of 
Proposition~\ref{katok} hold.
Let $\Phi_{p}=\{\phi_{\omega} \}_{\omega\in 
W_{p} }$ be a finite IFS on $[0,1]$ given by Proposition~\ref{katok}. Clearly the limit set of $\Phi_p$ is contained in $L(0)$.
Proposition~\ref{katok}(d) gives
\[\dim_{\rm H}L(0)\geq\dim_{\rm H}\Lambda(\Phi_p)>\dim_{\rm H}\Lambda(\Phi)-\varepsilon.\] Since $\varepsilon>0$ is arbitrary we obtain
 $\dim_{\rm H}L(0)\geq\dim_{\rm H}\Lambda(\Phi)$ as required.
In what follows we only consider the case $\alpha\in(0,\infty)$. The case $\alpha=\infty$ is covered by a slight modification of the following argument.

\subsection{Construction of a subset of the level set }\label{itineraries}
We construct a subset of the level set $L(\alpha)$ with large Hausdorff dimension by inserting digits into the limit set of $\Phi_p$. In order to
select positions of these digits,
we define
 sequences $\{n(k)\}_{k=0}^\infty$, $\{m(k)\}_{k=0}^\infty$ of non-negative integers inductively as follows. 
Start with $n(0)=0$. 
Given $n(k)$ for $k\geq0$ such that $n(k)+p< (k+p+1)^d$,
define \begin{equation}\label{n-ineq}n(k+1)=n(k)+m(k)p+1,\end{equation} where
$m(k)\geq2$ is the positive integer satisfying
\[n(k)+(m(k)-1)p\leq (k+p+1)^d<n(k)+m(k)p.\]
By \eqref{p} and \eqref{n-ineq} we have
\[n(k+1)+p= n(k)+m(k)p+p+1\leq (k+p+1)^d+2p+1<(k+p+2)^d,\]
which proves the assumption of induction.
The above definition implies
\begin{equation}\label{nk-diff}2\leq n(k)-(k+p)^d\leq p\ \text{ for every }k\geq1.\end{equation}
From \eqref{nk-diff} we obtain
\begin{equation}\label{limitzero}\lim_{k\to\infty}\frac{k}{n(k)}=0.\end{equation} 

Define a function
 $\tau\colon [2,\infty)\to(0,\infty)$ by $\tau(x)=x/\log\log x$. 
 Notice that
\[\lim_{k\to\infty}\frac{\tau(k^d)}{\tau( (k+1)^d)}=1.\]
We have $\lim_{x\to\infty}\tau'(x)=0$.
Applying the mean value theorem to $\tau$ and then using \eqref{nk-diff}, we get
\[\lim_{k\to\infty}\frac{\tau(n(k))-\tau(k^d)}{\tau((k+1)^d)}=0.\]
These two equalities shows $\lim_{k\to\infty}\tau(n(k))/\tau((k+1)^d)=1$, and 
\begin{equation}\label{cal-eq}\lim_{k\to\infty}\frac{\tau(n(k))}{\tau(n(k+1))}=1.
\end{equation}

Let $y\in \Lambda(\Phi_p)$.
Then $\{a_n(y)\}_{n\in\mathbb N}$ is a concatenation of elements of $W_p$.
We insert the digits $a_{n(k)}=\lfloor \alpha \tau(n(k))\rfloor$, $k\in\mathbb N$ into $\{a_n(y)\}_{n\in\mathbb N}$ to define a new sequence 
 \[\begin{split}&\ldots,a_{(m(0)+\cdots+m(k-1))p-1}(y),a_{(m(0)+\cdots+m(k-1))p}(y),\fbox{$a_{n(k)}$},a_{(m(0)+\cdots+m(k-1))p+1}(y),\ldots\\
  &\ldots,a_{(m(0)+\cdots+m(k))p-1}(y),a_{(m(0)+\cdots+m(k))p}(y),\fbox{$a_{n(k+1)}$},a_{(m(0)+\cdots+m(k))p+1}(y),\ldots\end{split}\]
Let $x(y)$ denote the point in $(0,1)\setminus\mathbb Q$ whose regular continued fraction expansion is given by this new sequence. Let $L_{p}(\alpha)$ denote the collection of these points:
\[L_{p}(\alpha)=\{x(y)\in (0,1)\setminus\mathbb Q\colon y\in \Lambda(\Phi_p)\}.\]

\begin{lemma}\label{contain-lem}
We have $L_{p}(\alpha)\subset L(\alpha).$\end{lemma}
\begin{proof}
Recall that $\Phi=\{\phi_i\}_{i\in I}$ is the $d$-decaying parabolic IFS on $[0,1]$.
Let $M$ denote the maximum of the finite subset of $I$ that consists of integers appearing in some element of $W_p$. Fix a sufficiently large integer $k_0\geq1$ such that $\alpha \tau(n(k_0))\geq M$.
Let $x\in L_{p}(\alpha)$.
For any $n\geq n(k_0)$ there exists $k\geq k_0$ such that \[n(k)\leq n<n(k+1).\]
Then 
\[\max\{a_1(x),\ldots,a_n(x)\}=\max\{M,a_{n(k)}(x)\}=\lfloor\alpha \tau(n(k) )\rfloor.\]
Since $\tau$ is monotone increasing, we have
\[\frac{\alpha \tau(n(k) ) -1}{\tau(n(k+1))}\leq\frac{\lfloor\alpha \tau(n(k))\rfloor}{\tau(n(k+1))}\leq\frac{\max\{a_1(x),\ldots,a_n(x)\}         }{\tau(n)}\leq \frac{\lfloor\alpha \tau(n(k) )\rfloor}{\tau(n(k)) }\leq\alpha.\]
As $n\to\infty$ we have $k\to\infty$, and from \eqref{cal-eq} obtain
\[\lim_{n\to\infty}
\frac{\max\{a_1(x),\ldots,a_n(x)\}}{\tau(n)}=\alpha,\]
namely $x\in L(\alpha)$ as required.
\end{proof}

\subsection{Completing the proof of Theorem~\ref{mainthm}}

The map $y\in \Lambda(\Phi_p)\mapsto x(y)\in L_{p}(\alpha)$ is bijective.
Let $h_\varepsilon\colon  L_{p}(\alpha)\to \Lambda(\Phi_p)$ denote the inverse of this map
that eliminates all the inserted digits $a_{n(k)}$, $k\in\mathbb N$.
\begin{prop}\label{prop-lip}
The map $h_\varepsilon$ is H\"older continuous with exponent $1/(1+\varepsilon)$.
\end{prop}

We finish the proof of Theorem~\ref{mainthm} subject to Proposition~\ref{prop-lip}. The next lemma is standard in dimension theory.

\begin{lemma}[{\cite[Proposition~3.3]{Fal14}}]
\label{Holder-lem}
Let $\Lambda\subset [0,1]$ and let $h\colon \Lambda\to [0,1]$ be H\"older continuous with exponent $\beta\in(0,1]$.
Then \[\dim_{\rm H}\Lambda\geq\beta\cdot\dim_{\rm H}h(\Lambda).\]
\end{lemma}

Proposition~\ref{prop-lip}  allows us to apply Lemma~\ref{Holder-lem} to the map $h_\varepsilon$.
Further, 
Lemma~\ref{contain-lem}, Lemma~\ref{Holder-lem} and Proposition~\ref{katok}(d) altogether yield
\[\dim_{\rm H}L(\alpha)\geq\dim_{\rm H}L_{p}(\alpha)\geq\frac{1}{1+\varepsilon}\dim_{\rm H}\Lambda(\Phi_{p})\geq\frac{1}{1+\varepsilon}\left(\dim_{\rm H}\Lambda(\Phi )-\varepsilon\right).\]
Since $\varepsilon>0$ is arbitrary, we obtain
$\dim_{\rm H}L(\alpha)\geq\dim_{\rm H}\Lambda(\Phi)$ as required in Theorem~\ref{mainthm}.

\subsection{The first technical lemma}
For a  proof of Proposition~\ref{prop-lip} we need two technical lemmas. The first one compares diameters of two fundamental intervals of different orders associated with $L_p(\alpha)$, one obtained from the other eliminating all the inserted digits. To give a precise statement we need some definitions. 
For each $k\in\mathbb N$,
let $A_{k}$ denote the set of $a_1\cdots a_{n(k+1)}\in I^{n(k+1)}$ for which the following two conditions hold for every  $0\leq j\leq k$:
\begin{itemize}
\item $a_{n(j)+1}\cdots a_{n(j+1)-1}\in I^{n(j+1)-n(j)-1}$ is a concatenation of elements of $W_p$;
\item $a_{n(j+1)}=\lfloor\alpha \tau(n(j+1) )\rfloor$.

\end{itemize}
 We set
\[B_k=\{
n(k)+mp\colon m=0,1,\ldots,m(k)\}.\]
Notice that
\[L_{p}(\alpha)=\bigcap_{k=0}^\infty\bigcup_{a_1\cdots a_{n(k+1)}\in A_{k}}J(a_1,\ldots,a_{n(k+1)}).\]

For each $n\in\mathbb N$ and $a_1\cdots a_{n}\in I^n$, let $\overline{a_1\cdots a_n}$ denote the word from $I$ 
obtained by eliminating from $a_1\cdots a_n$ the digits 
$a_{n(1)}$, $a_{n(2)},\ldots, a_{n(k)}$, 
$n(k)\leq n<n(k+1)$. 
Set
\[\overline{J}(a_1,\ldots,a_n)=\phi_{\overline{a_1\cdots a_n}}([0,1]).\]
Let $|\cdot|$ denote the Euclidean diameter of a set in $[0,1]$.
\begin{lemma}\label{compare}
There exists $k_1\in\mathbb N$ such that for any integer $k\geq k_1$,
 any $a_1\cdots a_{n(k+1)}\in A_{k}$ and any
$n\in B_k$
we have
\[|J(a_1,\ldots,a_n)|\geq|\overline{J}(a_1,\ldots,a_n)|^{1+\varepsilon}.\]
\end{lemma}

\begin{proof}

Take $\delta>0$ such that
\begin{equation}\label{delta-eq}\delta\log C\leq \frac{\varepsilon\gamma(1-\delta)}{4},\end{equation}
where $C\geq1$ is the constant in (A2).
Let $k\in\mathbb N$, $a_1\cdots a_{n(k+1)}\in A_k$
and let $n\in B_k$. We have
\[n(k)\leq n<n(k+1).\]
In view of \eqref{limitzero}, we assume $k$ is sufficiently large so that
\begin{equation}\label{tn}k+1\leq\delta n(k).\end{equation}

On the one hand, by the mean value theorem there exists $x_0\in[0,1]$ such that
\[|J(a_1,\ldots,a_n)|=|\phi'_{a_1\cdots a_n }(x_0)|.\]
For $x\in [0,1]$ we set 
\[R(x)=\begin{cases}|(\phi_{a_{n(k)+1}}\circ\!\!\!\!&\cdots\circ\phi_{a_n})'(x_0)|\ \text{ if }n(k)<n,\\
&1\quad\quad\quad\quad\quad\quad \text{ if }n(k)=n.\end{cases}\]
 The chain rule gives
\begin{equation}\label{compare-eq3}\begin{split}|J(a_1,\ldots,a_n)|=&\prod_{j=1}^{k}|(\phi_{a_{n(j-1)+1}}\circ\cdots\circ
\phi_{a_{n(j)-1} } )'(\phi_{a_{n(j)}\cdots a_n}(x_0))|\\
&\times\prod_{j=1}^{k}
|\phi'_{a_{n(j)} }(\phi_{a_{n(j)+1}}\circ\cdots\circ\phi_{a_n}(x_0))|\times R(x_0).\end{split}\end{equation}
By the mean value theorem and the chain rule, there exists $y_0\in[0,1]$ such that
\begin{equation}\label{compare-eq4}\begin{split}|\overline{J}(a_1,\ldots,a_n)|=&\prod_{j=1}^{k}|(\phi_{a_{n(j-1)+1} }\circ\cdots\circ\phi_{a_{n(j)-1} } )'(
\phi_{\overline{a_{n(j)}\cdots a_n} }(y_0))|\times R(y_0).\end{split}\end{equation}
From Proposition~\ref{katok}(b), $a_{n(j)-1}$ is not a parabolic index for every $1\leq j\leq k$. By (A2) we have
\[\frac{|(\phi_{a_{n(j-1)+1}}\circ\cdots\circ
\phi_{a_{n(j)-1} })'(\phi_{a_{n(j)}\cdots a_n}(x_0))|}{|(\phi_{a_{n(j-1)+1} }\circ\cdots\circ\phi_{a_{n(j)-1} })'(\phi_{\overline{a_{n(j)}\cdots a_n}}(y_0))|}\geq C^{-1}.\]
Since $k\geq k_1$ and the number of parabolic indices is finite by (A1), if $k_1$ is sufficiently large then $a_{n(k)}$ is not a parabolic index.
Since $n\in B_k$, $a_n$ is not a parabolic index and (A2) gives
\[\frac{R(x_0)}{R(y_0)}\geq C^{-1}.\]
Combining \eqref{compare-eq3}, \eqref{compare-eq4} and then applying the above two distortion estimates,
 we get
\begin{equation}\begin{split}\label{compare-eq1}\frac{|J(a_1,\ldots,a_n)|}{|\overline{J}(a_1,\ldots,a_n)|}&\geq  C^{-k-1}\prod_{j=1}^{k}
|\phi'_{a_{n(j)} }(\phi_{a_{n(j)+1}}\circ\cdots\circ\phi_{a_n}(x_0))|\\
&\geq C^{-k-1}\prod_{j=1}^{k} \min_{x\in[0,1]}|\phi'_{n(j)}(x)|.\end{split}\end{equation}
We estimate the two factors in the last expression separately. By $k+1\leq\delta n(k)\leq \delta n$ from \eqref{tn}, and then by \eqref{delta-eq} we have
\[C^{-k-1}\geq C^{-\delta n}\geq e^{ -\varepsilon\gamma
(n-\delta n)/4}.\]
Since
 $n\geq n(k)$ and $\Phi$ is $d$-decaying, 
 \eqref{nk-diff} implies
$k\leq 2n^{\frac{1}{d}}$, and
\[\prod_{j=1}^{k} 
\min_{x\in[0,1]}|\phi'_{n(j)}(x)|\geq \left(\frac{c}{n^{d}}\right)^{2n^{\frac{1}{d}} }\geq e^{ -\varepsilon\gamma
(n-\delta n)/4},\]
provided $k$ is sufficiently large. Plugging these two estimates into \eqref{compare-eq1} yields \begin{equation}\label{eq-s} \frac{|J(a_1,\ldots,a_n)|}{|\overline{J}(a_1,\ldots,a_n)|}\geq  e^{ -\varepsilon\gamma(n-\delta n)/2}\geq e^{ -\varepsilon\gamma(n-k)/2}.\end{equation}

On the other hand, since $\overline{a_1\cdots a_n}$ contains a concatenation of
$\lfloor(n-k)/p\rfloor$ elements of $W_p$ counted with multiplicity,  
Proposition~\ref{katok}(c) yields
\begin{equation}\label{compare-eq2}|\overline{J}(a_1,\ldots,a_{n})|^{\varepsilon}\leq
e^{-\varepsilon\gamma(n-k)/2},\end{equation}
provided $k$ is sufficiently large.
From
\eqref{eq-s} and \eqref{compare-eq2} we obtain the desired inequality.
\end{proof}

\subsection{The second technical lemma}
Before proceeding further, we introduce a positive constant
 \begin{equation}\label{K-def}K=\min_{\substack{\omega,\eta\in W_p\\ \omega\neq \eta}}\min\{|x-y|\colon x\in J(\omega),\ y\in J(\eta)\}.\end{equation}
Fundamental intervals are closed subintervals of $[0,1]$, and all fundamental intervals of order $p$ corresponding to words in $W_p$ are pairwise disjoint by
Proposition~\ref{katok}(a). Hence $K>0$ holds.

The second technical lemma for the proof of Proposition~\ref{prop-lip} gives a lower bound on the distance between two nearby points in $L_p(\alpha)$ in terms of the diameter of some fundamental interval. 
For a pair $(x,y)$ of distinct points in $L_{p}(\alpha)$, let $s(x,y)$ denote the maximal integer $n\geq0$
for which there exists a fundamental interval of order $n$ that contains $x$ and $y$. By 
Lemma~\ref{unif-decay}, $s(x,y)$ is well-defined.

\begin{lemma}\label{separate}
For any pair $(x,y)$ of distinct points in $L_{p}(\alpha)$ satisfying $s(x,y)\geq p$, there exist integers $k\geq0$, $m\geq0$ such that \[n(k)+mp\leq s(x,y),\ 
n(k)+mp\in B_k
\ \text{ and }\]
\[|J(a_1(x),\ldots,a_{n(k)+mp}(x))| \leq CK^{-1}|x-y|.\]\end{lemma}
\begin{proof}
 There exists $k\geq 0$
such that 
$n(k)\leq s(x,y)\leq n(k+1)-1.$
By the definition of $s(x,y)$ and $a_{n(k+1)}(x)=a_{n(k+1)}(y)$,
the second inequality is actually strict:
\[n(k)\leq s(x,y)<n(k+1)-1.\]
From the definition \eqref{n-ineq}, there exists $m\geq0$ such that
\begin{equation}\label{ineq}n(k)+mp\leq s(x,y)< n(k)+(m+1)p\leq n(k+1)-1.\end{equation}
Hence $n(k)+mp\in B_k$ holds.
Since $s(x,y)\geq p$ and $n(0)=0$, $k=0$ implies $m\geq1$.

The first and second inequalities in \eqref{ineq}
 together imply that \[a_i(x)=a_i(y)\ \text{ for }i=1,\ldots,n(k)+mp\ \text{ and }\]
 \[a_i(x)\neq a_i(y)\ \text{ for some }i\in\{n(k)+mp+1,\ldots,n(k)+(m+1)p\}.\] 
  Set \[a_i=a_i(x)\ \text{ for }i=1,\ldots, n(k)+mp,\ 
    \phi=\phi_{a_1\cdots a_{n(k)+mp}}\ \text{ and }\]
   \[\omega(z)=a_{n(k)+mp+1}(z)\cdots a_{n(k)+(m+1)p}(z)\in I^p \ \text{ for } z=x,y.\]
We have
$\phi^{-1}(x)\in J(\omega(x))$,
$\phi^{-1}(y)\in J(\omega(y))$ and 
$\omega(x)\neq  \omega(y)$.
By the definition of $L_{p}(\alpha)$, 
both $\omega(x)$ and $\omega(y)$ belong to $W_p$.
By (A2) and the definition of $K$ in \eqref{K-def}, we have
\[\begin{split}&\frac{|J(a_1,\ldots,a_{n(k)+mp})|}{|x-y|}\leq C\frac{|\phi^{-1}(J(a_1,\ldots,a_{n(k)+mp} ))|}{|\phi^{-1}(x)-\phi^{-1}(y)|}\leq CK^{-1},\end{split}\]
as required.
\end{proof}

\subsection{Proof of Proposition~\ref{prop-lip}}We set
\[K_1=\min_{\substack{x,y\in L_p(\alpha)\\ x\neq y }}\min\{|x-y|\colon s(x,y)\leq n(k_1)-1\},\]
where $k_1$ is the positive integer in Lemma~\ref{compare}.
Since $K>0$ we have $K_1>0$.
Let $(x,y)$ be a pair of distinct points in
$L_{p}(\alpha)$.
If $s(x,y)\leq n(k_1)-1$ then 
 $|x-y|\geq K_1,$ and so 
\begin{equation}\label{holder-ineq1}|h_\varepsilon(x)-h_\varepsilon(y)|\leq 1\leq K_1^{-1}|x-y|.\end{equation}
If $s(x,y)\geq n(k_1)\geq p$, then
let $k\geq 0$, $m\geq0$ be the integers for which the conclusion of Lemma~\ref{separate} holds. Since $x,y\in J(a_1(x),\ldots,a_{n(k)+mp}(x))$ we have
$h_\varepsilon(x),h_\varepsilon(y)\in\overline{J}(a_1(x),\ldots,a_{n(k)+mp}(x))$, and thus
\begin{equation}\begin{split}\label{holder-ineq2}|h_\varepsilon(x)-h_\varepsilon(y)|&\leq|
\overline{J}(a_1(x),\ldots,a_{n(k)+mp}(x))|\\
&\leq  |J(a_1(x),\ldots,a_{n(k)+mp}(x))|^{\frac{1}{1+\varepsilon}}\leq \left(CK^{-1} |x-y|\right)^{\frac{1}{1+\varepsilon}}.\end{split}\end{equation} 
To deduce the second inequality we have used
Lemma~\ref{compare}.
By \eqref{holder-ineq1} and \eqref{holder-ineq2},
$h_\varepsilon$ is H\"older continuous with exponent $1/(1+\varepsilon)$ as required.
\qed

\section{Verification of (B1) (B2) and examples}\label{verify-s}
In this last section we give a verifiable sufficient condition for (B1) (B2), and provide examples of $2$-decaying parabolic IFSs satisfying them. 
\begin{prop}\label{sat-prop}
Let $\Phi=\{\phi_i\}_{i\in I}$ be an infinite parabolic IFS such that $\phi_i$ is $C^2$ for each $i\in I$, and
\begin{equation}\label{dist-ineq}\sup_{i\in I}\max_{x\in[0,1]}|(\log| \phi'_{i}(x)|)'|<\infty.\end{equation}
Then (B1) (B2) hold.
\end{prop}
\begin{proof}
Let $f\colon\bigcup_{i\in I}\varDelta_i\to[0,1]$ be the Bernoulli map  associated with the IFS $\Phi$:
$\varDelta_i=\phi_i((0,1))$ and $f_i=f|_{\varDelta_i}$ for each $i\in I$.
We have $\phi_i\circ f_i(x)=x$ for all $x\in \varDelta_i$. Differentiating this equality twice
and rearranging the result gives 
\[\frac{f''_i(x)}{(f'_i(x))^2}=-\frac{\phi''_i(f_i(x))}{\phi'_i(f_i(x))}.\]
Since $i\in I$ and $x\in \varDelta_i$ are arbitrary, this equality and \eqref{dist-ineq} together imply that $f$ satisfies R\'enyi's condition \[\sup_{i\in I}\sup_{x\in\varDelta_i}\frac{|f''_i(x)|}{|f'_i(x)|^2}<\infty.\]
Condition \eqref{dist-ineq} also implies $D_1(\Phi)<\infty$. By Lemma~\ref{unif-decay}, $\Phi$ has uniform decay of fundamental intervals. It follows from \cite[Lemma~5.1]{JT} that $\Phi$ satisfies (B2).

Since $\Lambda(\Phi)\setminus\bigcap_{n=0}^\infty f^{-n}(\varDelta)$ is a countable set,
to verify (B1)
it suffices to check (i) (ii) (iii) (iv) in \cite[Proposition~5.2]{JT} for $f$. We have already shown that $f$ satisfies R\'enyi's condition. 
By (A1), 
(M3) in \cite{JT} holds and  
(i) in \cite[Proposition~5.2]{JT} is vacuous. 
 If we consider the first return map to the domain $\bigcup\{\varDelta_i\colon i\in I,\text{not parabolic}\}$, then all (ii) (iii) (iv) in \cite[Proposition~5.2]{JT} hold.\end{proof}

\subsubsection*{Example~1 (backward continued fraction)} Any irrational number $x$ in $(0,1)$ has a unique expansion of the form
\[
x=1-\confrac{1 }{a_{1}(x)} - \confrac{1 }{a_{2}(x)}  - \confrac{1 }{a_{3}(x)} -\cdots,\]
where $a_n(x)\in\mathbb N_{\geq2}$ for every $n\in\mathbb N$,
called {\it the backward continued fraction} \cite{IosKra02,R57}. 
The typical behavior of digits in this expansion in the sense of the Lebesgue measure on $(0,1)\setminus\mathbb Q$ is totally different from that of the regular continued fraction. 
For example, each integer $k\in\mathbb N$ typically
appears with positive definite asymptotic frequency $\frac{1}{\log2}\log\frac{(k+1)^2}{k(k+2)}$ in the latter, 
while in the former, any digit other than $2$ appears with 
asymptotic frequency zero.
For more details on typical behaviors of digits in the backward continued fraction, see \cite{A,AN,BS,IosKra02,T} for example.

The backward continued fraction is generated by the $2$-decaying parabolic IFS $\Phi=\{\phi_i\}_{i\in \mathbb N_{\geq2}}$ given by $\phi_i(x)=1-1/(x+i-1)$. 
The index $2$ is the only parabolic index: $\phi_2(0)=0$ and $\phi'_2(0)=1$. We have
$\Lambda(\Phi)=\{0\}\cup((0,1)\setminus\mathbb Q)$, and
 $x=\lim_{n\to\infty}\phi_{a_1(x)}\cdots\phi_{a_n(x)}(0)$ for all $x\in\Lambda(\Phi)$.
A direct calculation shows (A1). Condition (A2) follows from \eqref{dist-ineq} and the standard bounded distortion lemma near neutral fixed points as in the lemma below.
 Conditions (B1) (B2) follow from Proposition~\ref{sat-prop}.
\begin{lemma}[in the proof of {\cite[Lemma~5.3]{JT}}]
\label{scope'}
Let $f\colon[0,1)\to \mathbb R$ be a $C^{2}$ map satisfying
$f(0)=0$, $f'(0)=1$ and $f'(x)>1$ for all $x\in(0,1)$. There exists a constant $C>0$ such that
for every $n\in\mathbb N$ and all $x,y\in J_{n-1}$,
\[\log\frac{(f^n)'(x)}{(f^n)'(y)}\leq 
C|f^n(x)-f^n(y)|\sum_{i=0}^{n-1}\frac{|J_{i}|}{|J_0|},\]
where 
$q_0=1$, $f(q_{i+1})=q_{i}$ and $J_i=[q_{i+1},q_i)$ for $i=0,\ldots, n-1$.
\end{lemma}

\subsubsection*{Example~2 (even-integer continued fraction)}

Any irrational number $x$ in $(0,1)$ has a unique continued fraction expansion of the form
\[
x=\confrac{1 }{b_1 } + \confrac{\varepsilon_1 }{b_{2} } + \confrac{\varepsilon_2 }{b_{3} } +\cdots,\]
where $b_n$ is a positive even integer and $\varepsilon_n\in\{1,-1\}$
for all $n\in\mathbb N$, called {\it the even-integer continued fraction} \cite{Sch82,Sch84}. 
This expansion is generated by
 the $2$-decaying parabolic IFS
$\Psi=\{\psi_i\}_{i\in\mathbb N}$ given by
$\psi_i(x)=1/(i-x)$ for $i$ even and $\psi_i(x)=1/(x+i+1)$ for $i$ odd.
The index $2$ is the only parabolic index: $\psi_2(1)=1$ and $\psi'_2(1)=1$. We have 
$\Lambda(\Psi)=\{1\}\cup((0,1)\setminus\mathbb Q)$,  and for all $x\in\Lambda(\Psi)$,
$x=\lim_{n\to\infty}\psi_{a_1(x)}\cdots\psi_{a_n(x)}(0)$ and
\[
x=\confrac{1 }{b_{1}(x)} + \confrac{\varepsilon_1(x) }{b_{2}(x)} + \confrac{\varepsilon_2(x) }{b_{3}(x)} +\cdots,\]
where
$(b_n(x),\varepsilon_n(x))=(a_n(x),-1)$ if $a_n(x)$ is even and  
$(b_n(x),\varepsilon_n(x))=(a_n(x)+1,1)$ if $a_n(x)$ is odd.
In this expansion,
the typical behavior of digits in the sense of the Lebesgue measure on $(0,1)\setminus\mathbb Q$ is similar to that of the backward continued fraction.
A direct calculation shows (A1). By the same reasoning as for the backward continued fraction, one can check (A2).  Conditions (B1) (B2) follow from Proposition~\ref{sat-prop}.

\subsubsection*{Example~3} From a given infinite parabolic IFS, one can define a new one by removing countably many non-parabolic indices.
For example, 
for any proper infinite subset $I$ of $\mathbb N_{\geq2}$ (resp. of $\mathbb N$) containing $2$, the IFS
$\{\phi_i\}_{i\in I}$ from Example~1 (resp. the IFS $\{\psi_i\}_{i\in I}$ from Example~2) is a $2$-decaying parabolic IFS. Conditions (B1) (B2) follow from Proposition~\ref{sat-prop}.
\subsection*{Acknowledgments}
I thank the referee for its careful reading of the manuscript and giving valuable comments and suggestions.
I thank Johannes Jaerisch and Yuto Nakajima for fruitful discussions. 
This research was partially supported by the JSPS KAKENHI 25K21999, Grant-in-Aid for Challenging Research (Exploratory). 

\end{document}